\documentclass[11pt]{article}

\usepackage[a4paper,margin=30mm]{geometry}
\usepackage{amsmath,amssymb,amsthm,mathtools}
\usepackage{enumitem}
\usepackage{microtype}
\usepackage{hyperref}
\usepackage{mathrsfs}
\usepackage{aliascnt}
\usepackage{ragged2e}

\allowdisplaybreaks
\setlist{itemsep=2pt,topsep=4pt}

\hypersetup{
  colorlinks=true,
  linkcolor=blue,
  citecolor=blue,
  urlcolor=blue,
  pdftitle={Perfect Matching in k-Partite k-Uniform Hypergraphs},
  pdfsubject={Perfect matchings in balanced partite uniform hypergraphs},
  pdfkeywords={hypergraph matching, partite hypergraph, Dirac threshold, fractional matching, space barrier, stability}
}

\newtheorem{theorem}{Theorem}[section]
\newaliascnt{lemma}{theorem}
\newtheorem{problem}{Problem}[section]

\newtheorem{lemma}[lemma]{Lemma}
\aliascntresetthe{lemma}
\newaliascnt{proposition}{theorem}

\aliascntresetthe{proposition}
\newaliascnt{corollary}{theorem}
\newtheorem{corollary}[corollary]{Corollary}
\aliascntresetthe{corollary}
\theoremstyle{definition}
\newaliascnt{definition}{theorem}

\aliascntresetthe{definition}
\theoremstyle{remark}
\newaliascnt{remark}{theorem}

\aliascntresetthe{remark}
\usepackage[capitalise,noabbrev]{cleveref}

\newcommand{\E}{\mathbb E}
\newcommand{\Pp}{\mathbb P}
\newcommand{\Law}{\mathcal L}
\newcommand{\one}{\mathbf 1}
\newcommand{\cP}{\mathcal P}
\newcommand{\Wass}{d_{\mathrm W}}
\newcommand{\dd}{\,\mathrm d}

\newcommand{\nuast}{\nu^{*}}

\title{Perfect Matching in $k$-Partite $k$-Uniform Hypergraphs}
\date{}
\author{
Jie Han\thanks{School of Mathematics and Statistics,
Beijing Institute of Technology, Beijing, China.},
Hongliang Lu\thanks{School of Mathematics and Statistics,
Xi'an Jiaotong University, Xi'an, Shaanxi, China.},
Bin Wang\footnotemark[1] and
Feihong Yuan\footnotemark[2]
}

\begin{document}
\maketitle

\begin{abstract}
A balanced $k$-partite $k$-graph is a $k$-uniform hypergraph whose
vertex set is partitioned into $k$ classes of the same size and whose
edges meet every class in exactly one vertex.  Lo and Markstr\"om (2014)
determined the minimum vertex-degree threshold for perfect matchings
when $k=3$, and Lu, Wang and Yuan  recently determined it when $k=4$.
We prove the corresponding exact result for every fixed $k\ge5$ and
all sufficiently large class sizes.  The close case follows from the
general theorem of Lu, Wang and Yuan. For the non-closed case, we extend their stability result from the 3-partite setting to arbitrary partite uniformity, using the probability-tail rigidity theorem of Cao, Liu and Zhang, thereby replacing the earlier weighted lemma. 
\end{abstract}

\section{Introduction}

A \textit{\(k\)-uniform hypergraph}, or \textit{\(k\)-graph}, consists of a vertex set \(V\) and an edge set \(E \subseteq \binom{V}{k}\). A \textit{matching} in a hypergraph is a collection of pairwise disjoint edges, and a \textit{perfect matching} is one that covers all vertices of \(V\). For a set \(S \subseteq V\), let \(d_H(S)\) denote the number of edges containing \(S\). For \(0 \le \ell \le k-1\), the \textit{minimum \(\ell\)-degree} \(\delta_\ell(H)\) is defined as the minimum of \(d_H(S)\) over all \(\ell\)-subsets \(S \subseteq V\); in particular, \(\delta_0(H)=e(H)\) and \(\delta_1(H)=\delta(H)\).

A fundamental and challenging problem in extremal combinatorics is to determine minimum degree conditions that force the existence of a large matching, and in particular, a perfect matching. In the graph setting (i.e., \(k=2\)), Dirac's theorem~\cite{Dirac} states that every graph on
\(N\ge3\) vertices with minimum degree at least \(N/2\)
contains a Hamilton cycle, and hence a perfect matching
when \(N\) is even. For uniform hypergraphs, however, the problem becomes substantially more intricate and has received sustained attention. 
For \(N\ge k\) divisible by \(k\), let \(m_\ell(k,N)\) denote the smallest integer \(m\) such that every \(k\)-graph \(H\) on \(N\) vertices with \(\delta_\ell(H) \ge m\) contains a perfect matching.


For $1\le\ell\le k-1$, two standard obstructions compete.  The
\emph{space barrier} consists of all edges meeting a fixed set of size
slightly less than $N/k$, whereas a \emph{divisibility} or \emph{parity barrier} may
retain asymptotically one half of all edges.  These constructions motivated the asymptotic conjecture of
H\`an, Person, and Schacht~\cite{HanPersonSchacht}, which
asserted that
\[
 m_\ell(k,N)\sim
 \max\left\{
   \frac12,
   1-\left(1-\frac1k\right)^{k-\ell}
 \right\}
 \binom{N-\ell}{k-\ell}.
\]
The codegree case $\ell=k-1$ was determined exactly for every fixed
$k\ge3$ and all sufficiently large $N$ by R\"odl, Ruci\'nski and
Szemer\'edi \cite{RRS}.  For \(k\ge4\), Pikhurko~\cite{Pikhurko} obtained the asymptotically sharp
threshold in the range \(k/2\le\ell\le k-2\).
Treglown and Zhao~\cite{TreglownZhao1,TreglownZhao2} subsequently determined the corresponding exact thresholds.  In the more difficult range
$\ell<k/2$, H\`an, Person and Schacht~\cite{HanPersonSchacht} obtained the asymptotically sharp
vertex-degree threshold for $3$-graphs; the
exact threshold was then determined independently by K\"uhn, Osthus
and Treglown \cite{KuhnOsthusTreglown} and by Khan \cite{Khan3}.
Khan~\cite{Khan4} subsequently determined the exact vertex-degree threshold
for \(4\)-graphs. More recently, Frankl, Lu, Ma,
and Wu~\cite{FranklLuMaWu} determined the exact thresholds for
\((k,\ell)=(5,1)\) and \((6,2)\).
The asymptotic conjecture of H\`an, Person and Schacht \cite{HanPersonSchacht} has now
been resolved, using the connection established in \cite{AlonEtAl,FJ19}, through recent work on Feige’s conjecture by Fu et al.~\cite{FHWYZZ} and, independently, by Nie and Wei~\cite{NieWei}.
At the exact level, Han, Lu, Wang and Yuan~\cite{HLWY}
settled all remaining cases, thereby determining \(m_\ell(k,N)\)
for every fixed \(k\ge3\) and \(1\le\ell\le k-1\) and all
sufficiently large \(N\) divisible by \(k\).
Their result resolves the conjecture of Treglown and Zhao \cite{TreglownZhaoNote}
on exact minimum degree thresholds, including the vertex-degree
case conjectured earlier by K\"uhn, Osthus and Treglown \cite{KuhnOsthusTreglown}.
We also note that the connections among the Erd\H{o}s matching problem, the perfect matching threshold problem, and fractional matching problems have been systematically explored in \cite{AlonEtAl,KuhnOsthusTownsend,TreglownZhaoNote}.

\subsection{Matchings in $k$-partite $k$-graphs}
A $k$-partite $k$-graph has a vertex partition
$V_1\dot\cup\cdots\dot\cup V_k$ such that every edge contains exactly one
vertex from each class.  A set is \emph{legal} if it meets every class
in at most one vertex.  We call the hypergraph \emph{balanced} if
$|V_1|=\cdots=|V_k|$.  For $0\le\ell\le k-1$, the minimum legal
$\ell$-degree $\delta_\ell(H)$ is the minimum of $d_H(S)$ over all
legal $\ell$-sets $S$.  Define $m'_\ell(k,n)$ to be the least integer
$m$ such that every balanced $k$-partite $k$-graph $H$ with $n$ vertices in each class and
$\delta_\ell(H)\ge m$ contains a perfect matching.

K\"uhn and Osthus~\cite{KuhnOsthus}  established an
asymptotically sharp minimum codegree condition for perfect
matchings in balanced \(k\)-partite \(k\)-graphs. Aharoni, Georgakopoulos and Spr\"ussel~\cite{AharoniGeorgakopoulosSprussel} later showed
that $m'_{k-1}(k,n)\le n/2+1$. 
For every fixed \(k\ge3\) and all sufficiently large \(n\),
Lu, Wang, and Yu~\cite{LuWangYuCodegree} characterized
the balanced \(k\)-partite \(k\)-graphs with
\(\delta_{k-1}(H)\ge\lfloor n/2\rfloor\) and no perfect
matching, thereby determining \(m'_{k-1}(k,n)\) exactly.  For almost perfect matchings, Han, Zang, and
Zhao~\cite{HanZangZhao} and, independently, Lu, Wang,
and Yu~\cite{LuWangYuAlmost} proved that, for every fixed
\(k\ge3\) and all sufficiently large \(n\), the condition
\(\delta_{k-1}(H)\ge n/k\) guarantees a matching of size
\(n-1\). Pikhurko~\cite{Pikhurko} proved an Ore-type condition
involving the minimum degrees of complementary types
of legal tuples. In particular, his result implies $m'_\ell(k,n)\sim \frac12 n^{k-\ell}$ for every fixed \(k\ge3\) and \(k/2\le\ell<k\), as
\(n\to\infty\).
For the vertex-degree problem, Aharoni, Georgakopoulos, and Spr\"ussel~\cite{AharoniGeorgakopoulosSprussel} proposed the following problem. 
\begin{problem}\label{AGS}
Is it true that if $ m'_1(k,n)<(1-1/e)n^{k-1}$?
\end{problem}
Lo and Markstr\"om  \cite{LM} determined
$m'_1(3,n)$ exactly for all sufficiently large $n$.
Lu, Wang and Yuan~\cite{LWY}  subsequently determined $m'_1(4,n)$  for all sufficiently large \(n\).
As part of their proof, they established the close
case for every fixed \(k\ge4\).

In this paper, we determine  \(m'_1(k,n)\) for every
fixed \(k\ge3\) and all sufficiently large \(n\), and thus answer Problem \ref{AGS} in the affirmative.
Our proof builds on the framework of Lu, Wang, and
Yuan~\cite{LWY}. A key ingredient in their non-close
argument is a weighted stability lemma for
\(3\)-partite \(3\)-graphs.
In this paper we establish an analogous weighted stability lemma for balanced \((k-1)\)-partite \((k-1)\)-graphs arising
as vertex links in \(k\)-partite \(k\)-graphs. Our proof builds on recent breakthroughs on Feige's conjecture and its connection to fractional matchings.  

\subsection{Extremal constructions and main result}

We adopt the notation from Lo, Markstr\"{o}m \cite{LM} and Lu, Wang, Yuan \cite{LWY}. Let
$V_1,\ldots,V_k$ be pairwise disjoint sets of size $n$.  For integers
$0\le d_i\le n$, choose $W_i\subseteq V_i$ with $|W_i|=d_i$, and let \(W:=W_1\cup\cdots\cup W_k\).
Define \(H_k(n;d_1,\ldots,d_k)\) to be the \(k\)-partite
\(k\)-graph with vertex classes \(V_1,\ldots,V_k\) and edge set
\[
 E\bigl(H_k(n;d_1,\ldots,d_k)\bigr)
 :=
 \left\{
   \{v_1,\ldots,v_k\}:
   v_i\in V_i\text{ for each }i\in[k],\
   \{v_1,\ldots,v_k\}\cap W\ne\varnothing
 \right\}.
\]
For each integer $1\le m\le kn$, define
\begin{equation}
\label{eq:dim}    
d_i(m):=\left\lfloor\frac{m+i-1}{k}\right\rfloor
\end{equation}
for each $i\in[k]$ and $ H_k(n;m):=H_k(n;d_1(m),\ldots,d_k(m))$.
Thus $\sum_i d_i(m)=m$.  Let $H_k^0(n;m)$ be obtained from
$H_k(n;m)$ by deleting every edge contained entirely in
$W_1\cup\cdots\cup W_k$.
Write $m=rk+s$, where $r\ge0$ and
$s\in[k]$, and put $t=\lfloor m/(r+1)\rfloor$.  Define
\[
 a_i=
 \begin{cases}
  r+1,&1\le i\le t,\\
  m-(r+1)t,&i=t+1,\\
  0,&t+2\le i\le k,
 \end{cases}
\]
where the middle line is omitted when $t=k$.  Set
$H'_k(n;m):=H_k(n;a_1,\ldots,a_k)$ and $d_k(n,m):=\delta_1\bigl(H'_k(n;m)\bigr)$.
A direct calculation gives
\begin{equation}\label{eq:dk-general}
 d_k(n,m)
 =n^{k-1}
  -n^{\max\{0,k-t-1\}}
   (n-r-1)^{t-1}
   \bigl(n-m+(r+1)t\bigr)^{\min\{1,k-t\}}.
\end{equation}
Every edge of $H'_k(n;n-1)$ meets a fixed set of $n-1$ vertices, so
it contains no perfect matching.  It follows that
$m'_1(k,n)\ge d_k(n,n-1)+1$.
Our main result establishes that this lower bound is sharp for all sufficiently large \(n\).

\begin{theorem}\label{thm:main}
For every fixed integer $k\ge4$, there is $n_0=n_0(k)$ such that, for
all $n\ge n_0$,
\[
 m'_1(k,n)=d_k(n,n-1)+1.
\]
\end{theorem}
Given \(\varepsilon>0\), a balanced \(k\)-partite \(k\)-graph
\(H\) with \(n\) vertices in each class is said to be
\(\varepsilon\)-\emph{contained} in \(H_k^0(n;n)\) if there exists
a copy \(F\) of \(H_k^0(n;n)\) on \(V(H)\), respecting the
vertex partition up to a permutation of the classes, such that
$ |E(F)\setminus E(H)|\le\varepsilon n^k$.




The proof of \Cref{thm:main} builds on the stability framework of R\"odl, Ruci\'nski and Szemer\'edi \cite{RRS}, and for the $k$-partite case, by Lu, Wang and Yuan \cite{LWY}. 
Fortunately, the case $H$ is close to the extremal example is already resolved in \cite{LWY}.
For the non-close case, we first establish a perfect fractional matching theorem, whose proof relies on a weighted stability result for \((k-1)\)-partite link hypergraphs. 
The latter is derived from the recent breakthroughs on perfect fractional matchings by Cao et al.~\cite{CLZ}, which in turn is based on Feige's conjecture on sum of independent
nonnegative random variables by Fu et.al.~\cite{FHWYZZ} and by Nie and Wei~\cite{NieWei}. 
We then convert the perfect fractional matching(s) to an almost perfect matching, which together with an absorbing lemma by Lo and Markstr\"om yields a perfect matching.

\medskip
\noindent\textbf{Organization.}
 In \Cref{sec:fractional}, we prove the non-close perfect fractional matching theorem, along with the weighted partite stability lemma and its probabilistic proof. In \Cref{sec:nonclose}, we combine fractional matchings, the absorption technique, and the nibble method to resolve the non-close case. The close-case result of Lu, Wang, and Yuan is presented in \Cref{sec:completion}. Finally, in \Cref{sec:completion}, we merge the close and non-close cases to complete the proof of \Cref{thm:main}.

\section{Perfect fractional matchings}\label{sec:fractional}

It is well-known now in dense hypergraphs, the existence of almost perfect matchings is essentially equivalent to the existence of perfect fractional matchings.
This allows us to study the problem via linear programming tools for fractional matchings.
Let $H$ be a $k$-graph.  
A \emph{fractional matching} in $H$ is a function
$f:E(H)\to[0,1]$ such that $\sum_{e\ni v}f(e)\le1$ for every vertex $v$.
Its \emph{weight} is $\sum_{e\in E(H)}f(e)$, and it is \emph{perfect} if its weight is
$|V(H)|/k$.  
The \emph{fractional matching number} of $H$ is $\nu^*(H):=
\max\{
\sum_{e\in E(H)}f(e):
f \text{ is a fractional matching of }H\}$.
A \emph{fractional vertex cover} is a function
$g:V(H)\to[0,1]$ such that $\sum_{v\in e}g(v)\ge1$ for every edge $e$. 
The \emph{fractional vertex-cover number} of $H$ is $\tau^*(H):=
\min\{
\sum_{v\in V(H)}g(v):
g \text{ is a fractional vertex cover of }H\}$.
Linear programming duality gives $\nu^*(H)=\tau^*(H)$.

The main result of this section is the following theorem on fractional matchings, whose $k=4$ case was proved by Lu, Wang and Yuan \cite{LWY}.

\begin{theorem}\label{thm:fractional-nonclose}
Fix $k\ge3$ and $\varepsilon>0$.  There are
$\gamma=\gamma(k,\varepsilon)>0$ and $n_0$ such that the following
holds for $n\ge n_0$.  If $H$ is a balanced $k$-partite $k$-graph with \(n\) vertices in each class and
\[
\delta_1(H) \geq \left(1-\left(\frac{k-1}{k}\right)^{k-1}-\gamma\right)n^{k-1},
\]
and $H$ is not $\varepsilon$-contained in $H_k^0(n;n)$, then $H$ has a
perfect fractional matching.
\end{theorem}

We use the notation $a=b\pm c$ to indicate $b-c\le a\le b+c$. The key new ingredient for the proof of \Cref{thm:fractional-nonclose} is the following stability result, which indeed will be applied to the link $(k-1)$-partite $(k-1)$-graph of a vertex.

\begin{lemma}
\label{thm:weighted-stability}
Fix $r\ge2$ and $0<\xi<1/(2r)$.  
There is $\eta=\eta(r,\xi)>0$ such that the following holds for every $n$.
Let $L$ be a balanced $r$-partite $r$-graph with vertex classes \(V_1,\ldots,V_r\), each of size \(n\), and let $g:V(L)\to[0,1]$ be a fractional vertex cover.  
Suppose $ g(V(L))+\max_i g(V_i)\le n$ and  $e(L)\ge\left(1-\left(\frac r{r+1}\right)^r-\eta\right)n^r$.
Then for every $i\in[r]$ there are disjoint $W_i,X_i\subseteq V_i$ satisfying
\[
\begin{aligned}
|W_i| &= \left(\frac{1}{r+1}\pm\xi\right)n,
&\qquad g(w) &\ge 1-\xi \quad (w\in W_i),\\
|X_i| &= \left(\frac{r}{r+1}\pm\xi\right)n,
&\qquad g(x) &\le \xi \quad (x\in X_i).
\end{aligned}
\]
\end{lemma}

A slight mysterious assumption in the lemma is the property of the fractional cover $ g(V(L))+\max_i g(V_i)\le n$, which is indeed natural and crucial to control the fractional cover when one part contributes significant weights. 
The stability structure will be extended from this link hypergraph to the entire $k$-graph by the additional property of the fractional covers.

We first prove \Cref{thm:fractional-nonclose} using
\Cref{thm:weighted-stability}; the proof of \Cref{thm:weighted-stability} is postponed to
\Cref{subsec:weighted-proof}.

\subsection{Proof of \Cref{thm:fractional-nonclose}}
\noindent\textbf{Proof outline.}
To prove \Cref{thm:fractional-nonclose}, we assume that \(H\) has no perfect fractional matching and use
linear programming duality to find
a vertex link satisfying the hypotheses of
\Cref{thm:weighted-stability}.
The minimum vertex-degree condition allows us to recover
the corresponding structure in the remaining class.
We then use this structure and the degree condition to show
that \(H\) is \(\varepsilon\)-contained in \(H_k^0(n;n)\),
contradicting the hypothesis.
\begin{proof}
Choose $0<\gamma\ll\xi\ll\varepsilon,k$.
Assume that $H$ has no perfect fractional matching.  Let $f$ be a
maximum fractional matching and $g$ be a minimum fractional vertex cover.
By linear programming duality, $f(E(H))=g(V(H))<n$.
Relabel the classes so that $g(V_1)=\max_i g(V_i)$.  
We claim that there exists some $v_0\in V_1$ such that $g(v_0)=0$. 
Indeed, otherwise $g(v)>0$ for every $v\in V_1$, so $\sum_{e\ni v} f(e)=1$ for every $v\in V_1$ by complementary of slackness. 
Summing over all $v\in V_1$ and using the fact that each edge of $H$ contains exactly one vertex of $V_1$, we obtain $n=\sum_{v\in V_1}\sum_{e\ni v}f(e)=\sum_{e\in E(H)}f(e)=f(E(H))$, a contradiction.

Let $L=N_H(v_0)$ be the link of $v_0$ on $V_2,\ldots,V_k$. 
If $e$ is an edge of $L$, then $e\cup\{v_0\}\in E(H)$ and thus $g(v_0)+\sum_{v\in e}g(v)\ge1$.
Since $g(v_0)=0$, we obtain that $\sum_{v\in e}g(v)\ge1$ and hence $g$ is also a vertex cover of $L$.
Note that $V(L)=V(H)\setminus V_1$ and $g(V(L))=g(V(H))-g(V_1)\le g(V(H))-\max_{2\le i\le k}g(V_i)$.
Thus, $ g(V(L))+\max_{2\le i\le k}g(V_i)
 \le g(V(H))<n$.
Also
\[
 e(L)=d_H(v_0)\ge\left(1-\left(\frac{k-1}{k}\right)^{k-1}-\gamma\right)n^{k-1}.
\]
Applying
\Cref{thm:weighted-stability} to $L$ with parameter
$\xi$, we obtain disjoint sets $W_i,X_i\subseteq V_i$ for
$i\in[2,k]$ such that
$|W_i| = (\frac{1}{k}\pm\xi)n,|X_i| = (\frac{k-1}{k}\pm\xi)n$, and $ g(w) \ge 1-\xi$ for every $w\in W_i$, $g(x) \le \xi$ for every $x\in X_i$.

Next, we recover a similar structure in $V_1$.  
Since $|W_i| = (\frac{1}{k}\pm\xi)n$ and $g(w)\ge1-\xi$ for every $w\in W_i\subseteq V_i$ and every $i\in[2,k]$, we have $g(V_i)\ge g(W_i)\ge(1-\xi)|W_i|\ge(1-\xi)(\frac1k-\xi)n$. 
Thus, $\sum_{i=2}^k g(V_i)\ge (k-1)(1-\xi)\left(\frac{1}{k}-\xi\right)n\ge \left(\frac{k-1}{k}-2k\xi\right)n$ and $g(V_1)\le(\frac{1}{k}+2k\xi)n$.  

Define
$ W_1:=\{u\in V_1:g(u)\ge1-(k-1)\xi\}$.
Because $\xi$ is sufficiently small,
\begin{equation}
\label{ine:W_1upper}
 |W_1|\le\frac{1/k+2k\xi}{1-(k-1)\xi}n
 \le(1/k+6k\xi)n.          
\end{equation}
For a lower bound of $|W_1|$, 
take a $k$-tuple $e'=\{v_1,x_2,\ldots,x_k\}$ where $v_1\in V_1\setminus W_1$ and $x_i\in X_i$ for $i\in[2,k]$.
Note that $e'\notin E(H)$ because $g(v_1)+g(x_2)+\cdots+g(x_k)<1$.
Set $q:=(k-1)/k$, $p:=1/k$, $c_k:=1-q^{k-1}$ and $w_1:=|W_1|/n$. Then we have
\[
 (c_k-\gamma)n^{k-1}
 \le d_H(x_k)
 \le n^{k-1}-(1-w_1)(q-\xi)^{k-2}n^{k-1}.
\]
Thus, $1-w_1\le\frac{q^{k-1}+\gamma}{(q-\xi)^{k-2}}$.
Let $F(x,y)=\frac{q^{k-1}+y}{(q-x)^{k-2}}$, it is easy to see that $F(x,y)$ is continuously differentiable near $(0,0)$ and $F(0,0)=q$.
Thus there is a constant $C_k$ such that, when $\gamma\le\xi$ are sufficiently small,
$1-w_1\le F(\xi,\gamma)\le  q+C_k\xi$ and thus $|W_1|\ge(p-C_k\xi)n$.
Combining with \eqref{ine:W_1upper} and the conclusions for
$W_2,\ldots,W_k$, we obtain that for each $i\in[k]$,
\begin{equation}
\label{ine:W_i}
|W_i| = \left(\frac{1}{k}\pm C_k\xi\right)n.              
\end{equation}

Recall that $d_i(m)=\left\lfloor\frac{m+i-1}{k}\right\rfloor$ in \eqref{eq:dim}.
For each $i\in[k]$, choose $W_i^0\subseteq V_i$ with
$|W_i^0|=d_i(n)$ and with symmetric difference from $W_i$ as small
as possible, and put $U_i=V_i\setminus W_i^0$.  Since
$|d_i(n)-pn|=O_k(1)$, \eqref{ine:W_i} gives
$ |W_i^0\triangle W_i|\le 2C_k\xi n$, as $n$ is sufficiently large.                

Define $Z_1:=V_1\setminus W_1$ and $Z_i:=X_i$, $2\le i\le k$.
The sets $W_i$ and $X_i$ are disjoint, and their sizes add to $(1\pm2\xi)n$.  
Therefore, 
\begin{equation}
\label{ine:sym}
\begin{aligned}
|U_1\triangle Z_1|&=|(V_1\setminus W_1^0)\triangle (V_1\setminus W_1)|=|W_1^0\triangle W_1|\le2C_k\xi n,\\ 
 |U_i\triangle Z_i|&\le|U_i\triangle (V_i\setminus W_i)|+|(V_i\setminus W_i)\triangle X_i|\le |W_i^0\triangle W_i|+|(V_i\setminus W_i)\setminus X_i|\le 3C_k\xi n,
 \end{aligned}
\end{equation}
where $i\in[2,k]$.
Moreover, the box $Z_1\times\cdots\times Z_k$ is independent in $H$:
since for every $k$-tuple in that box, $g(z_1)+\cdots+g(z_k)
 <1-(k-1)\xi+(k-1)\xi=1$.

Let $H^0$ be the copy of $H_k^0(n;n)$ determined by
$W_1^0,\ldots,W_k^0$.  
Towards the final contradiction let us upper bound $|E(H^0)\setminus E(H)|$.
Note that every edge of $H^0$ contains a vertex in some
$U_i$.  Fix $i\in [k]$ and $u\in U_i\cap Z_i$.  Since $u\in U_i$, an edge of $H$
through $u$ that is not in $H^0$ must choose all its other vertices
from the sets $U_j$.  By the independence of the $Z$-box and \eqref{ine:sym},
the number of such edges is at most $\sum_{j\ne i}|U_j\setminus Z_j|n^{k-2}\le 3kC_k\xi n^{k-1}$.
Thus, $|N_{H}(u)\setminus N_{H^0}(u)|\le 3kC_k\xi n^{k-1}$.
Also, $d_{H^0}(u)=n^{k-1}-\prod_{j\ne i}|U_j|
          \le c_kn^{k-1}+\gamma n^{k-1}$.
Therefore, as $\gamma\ll\xi$ we obtain
\[
 |N_{H^0}(u)\setminus N_H(u)|=d_{H^0}(u)-d_H(u)+|N_{H}(u)\setminus N_{H^0}(u)|\le 4kC_k\xi n^{k-1}.   \]
For $u\in U_i\setminus Z_i$ we use the trivial bound $n^{k-1}$ and 
by \eqref{ine:sym} there are at most $3C_k\xi n$ such vertices $u$.
Hence we get 
\[
 \sum_{u\in U_i}|N_{H^0}(u)\setminus N_H(u)|
 \le 4kC_k \xi n^{k-1}|U_i| + n^{k-1} 3C_k\xi n\le 5kC_k \xi n^k,  
\]
and summing over $i\in [k]$ gives $|E(H^0)\setminus E(H)|\le  5k^2C_k \xi n^k\le \varepsilon n^k$, as $\xi\ll \varepsilon$.
Thus $H$ is $\varepsilon$-contained in $H_k^0(n;n)$, a contradiction.
\end{proof}

\subsection{Proof of \Cref{thm:weighted-stability}}
\label{subsec:weighted-proof}

In this subsection we prove \Cref{thm:weighted-stability}.
We first recall a sharp tail inequality and its equality
characterization, due to Cao, Liu, and
Zhang~\cite[Theorem~2.5]{CLZ}.
For a random variable \(Z\), let \(\Law(Z)\) denote its
probability distribution.
For $a\in\mathbb R$, let $\delta_a$ denote the unit point mass at $a$.
Thus, \(\Law(Z)=\delta_a\) means that \(Z=a\) almost surely.
For \(a<b\) and \(0<x<1\), the identity $\Law(Z)=(1-x)\delta_a+x\delta_b$ means that $Z$ follows a \emph{two-point distribution} where $\mathbb P(Z=a)=1-x$ and $\mathbb P(Z=b)=x$.

\begin{theorem}[Cao, Liu and Zhang~\cite{CLZ}]
\label{thm:clz}
Let $r\ge1$ and $T\ge r+1$.  If $Y_1,\ldots,Y_r$ are independent,
nonnegative random variables with $\E Y_i\le1$, then
\[
 \Pp\left(\sum_{i=1}^rY_i\ge T\right)
 \le 1-\left(1-\frac1T\right)^r.
\]
Equality holds if and only if, for every $i\in [r]$, $\mathcal L(Y_i)=(1-1/T)\delta_0 + (1/T)\delta_T$.
\end{theorem}

We next derive a consequence of this theorem that will be used
in the proof of \Cref{thm:weighted-stability}.
For $r\ge2$, let
$p_r:=\frac1{r+1}$, $q_r:=\frac r{r+1}$,
$b_r:=q_r^r$, $\beta_r:=q_r\delta_0+p_r\delta_1$.

\begin{corollary}
\label{lem:prob-rigidity}
Let $Z_1,\ldots,Z_r$ be independent random variables taking values in $[0,1]$, and set $a_i=\E Z_i$, $A=\sum_i a_i$, $M=\max_i a_i$.
If $A+M\le1$, then
\[
 \Pp\left(\sum_iZ_i<1\right)\ge b_r.
\]
Equality holds if and only if $\Law(Z_i)=\beta_r$ for every $i$.
\end{corollary}

\begin{proof}
If $M=0$, then every $Z_i$ vanishes almost surely and the assertion is strict.  
Suppose $M>0$ and define $Y_i:=Z_i/M+1-a_i/M$, and $T:=r+(1-A)/M$.
Then $Y_i\ge0$, $\E Y_i=1$, and $T\ge r+1$.  
Moreover, we have
$\sum_iY_i\ge T$ if and only if $\sum_iZ_i\ge1$.
By \Cref{thm:clz}, we obtain
\[
\Pp\left(\sum_iZ_i<1\right)
 \ge\left(1-\frac1T\right)^r
 \ge\left(1-\frac1{r+1}\right)^r=b_r.
\]
If equality holds, strict monotonicity in $T$ gives $T=r+1$, and the
equality statement of \Cref{thm:clz} gives
\[
  \Law(Y_i)=q_r\delta_0+p_r\delta_{r+1}.
\]
Since $Y_i\ge1-a_i/M\ge0$ and $\mathbb{P}(Y_i=0)=q_r>0$, we must have $1-a_i/M=0$.
Hence $a_i=M$ for every $i$. Consequently,
$A=rM$ and $T=1/M=r+1$, so
$M=1/(r+1)$.  
Since $Y_i=Z_i/M+1-a_i/M$, we have 
$Z_i=(Y_i-1)M+a_i=MY_i= Y_i/(r+1)$ and thus $\Law(Z_i)=q_r\delta_0+p_r\delta_1=\beta_r$.  
The converse is immediate.
\end{proof}

Let $\Wass$ denote the $1$-Wasserstein distance on probability laws on $[0,1]$. 
Explicitly, for $\mu,\nu\in\cP([0,1])$, let $\Pi(\mu,\nu)$ be the set of Borel probability measures on $[0,1]^2$ whose first and second marginals are $\mu$ and $\nu$, respectively, and define
\[
 \Wass(\mu,\nu)
 :=\inf_{\pi\in\Pi(\mu,\nu)}
   \int_{[0,1]^2}|x-y|\,\dd\pi(x,y).
\]

We derive the following stability version of \Cref{lem:prob-rigidity} by a compactness argument.
Note that similar standard proofs appeared in \cite{CLZ}.

\begin{lemma}[Compactness stability]\label{lem:prob-stability}
For every $r\ge2$ and every $\rho>0$, there is $\eta>0$ such that the following holds.  If independent  $[0,1]$-valued variables $Z_1,\ldots,Z_r$ satisfy $\sum_i\E Z_i+\max_i\E Z_i\le1$ and $\Pp\left(\sum_iZ_i<1\right)\le b_r+\eta$,
then
\begin{equation}
\label{eq:maxrho}
 \max_i\Wass(\Law(Z_i),\beta_r)<\rho.
\end{equation}
\end{lemma}

\begin{proof}
Otherwise there is a sequence of counterexamples for which the displayed
probability tends to $b_r$ and each member violates \eqref{eq:maxrho}.  
By compactness of $\cP([0,1])$, we may pass to a
subsequence on which every marginal law converges weakly, say to
$\mu_i$.  The product laws converge weakly to the limit
$\mu:=\mu_1\otimes\cdots\otimes\mu_r$, and the mean constraint passes to the
limit, so that \Cref{lem:prob-rigidity} applies to independent $[0,1]$-valued variables with laws $\mu_1,\dots, \mu_r$.  
As the set
\[
 \Lambda:=\{(z_1,\ldots,z_r):z_1+\cdots+z_r<1\}
\]
is open, Portmanteau's Theorem gives that its limiting product measure $\int_\Lambda d\mu$ is at
most $b_r$.  By \Cref{lem:prob-rigidity} it is at least $b_r$; hence
$\int_\Lambda d\mu = b_r$.
The equality classification yields $\mu_i=\beta_r$
for every $i$.  On the compact interval $[0,1]$, weak convergence is
equivalent to convergence in $1$-Wasserstein distance, contradicting
\eqref{eq:maxrho}. 
\end{proof}

Now we derive \Cref{thm:weighted-stability} from \Cref{lem:prob-stability}.

\begin{proof}[Proof of \Cref{thm:weighted-stability}]
Choose a uniformly random vertex $x_i\in V_i$, independently in the $r$ classes, and put $Z_i=g(x_i)$. We first need to verify the conditions of \Cref{lem:prob-stability}. 
It is easy to see that $\mathbb EZ_i=\frac1n\sum_{v\in V_i}g(v)=\frac1ng(V_i)$ and thus $\sum_{i\in[r]}\mathbb EZ_i=\frac1ng(V(L))$, $\max_i\mathbb EZ_i=\frac1n\max_ig(V_i)$.
Since $g(V(L))+\max_ig(V_i)\le n$, we have $\sum_{i\in[r]}\mathbb EZ_i+\max_i\mathbb EZ_i\le1$.
By the definition of fractional vertex cover, we obtain that if $\{x_1,\ldots,x_r\}\in E(L)$, then $Z_1+\cdots+Z_r\ge1$.
Choose $\rho=\xi^2/4$ and let $\eta$ be given by \Cref{lem:prob-stability}. 
Thus $\mathbb P(\sum_i Z_i<1)\le\mathbb P(\{x_1,\ldots,x_r\}\notin E(L))\le1-\frac{e(L)}{n^r}\le b_r+\eta$.
By \Cref{lem:prob-stability}, we have
$\Wass(\Law(Z_i),\beta_r)<\rho$ for every $i$.
Write \(\mu_i:=\Law(Z_i)\).
For any \(1\)-Lipschitz function
\(\varphi:[0,1]\to\mathbb R\) and any coupling
\(\pi\in\Pi(\mu_i,\beta_r)\), the marginal conditions give
\[
\begin{aligned}
  \left|
    \int_{[0,1]}\varphi\,\dd\mu_i
    -
    \int_{[0,1]}\varphi\,\dd\beta_r
  \right|
  =
  \left|
    \int_{[0,1]^2}
    \bigl(\varphi(x)-\varphi(y)\bigr)\,\dd\pi(x,y)
  \right|
  \le
  \int_{[0,1]^2}|x-y|\,\dd\pi(x,y).
\end{aligned}
\]
Taking the infimum over all such couplings yields
\[
  \left|
    \E\varphi(Z_i)
    -
    \int_{[0,1]}\varphi\,\dd\beta_r
  \right|
  \le \Wass(\mu_i,\beta_r)<\rho.
\]
Now the functions  \(z\mapsto z\) and
\(z\mapsto\min\{z,1-z\}\)
are both \(1\)-Lipschitz on \([0,1]\).
Since \(\beta_r=q_r\delta_0+p_r\delta_1\), we have
\[
  |\E Z_i-p_r|<\rho,
  \qquad
  \E\min\{Z_i,1-Z_i\}<\rho.
\]
Let $A_i=\{v\in V_i:g(v)>1/2\}$,
$W_i=\{v\in V_i:g(v)\ge1-\xi\}$, and
$X_i=\{v\in V_i:g(v)\le\xi\}$.  The pointwise identity $|\one_{\{z>1/2\}}-z|=\min\{z,1-z\}$
implies that   
\begin{equation}
\label{ine:a_i}
||A_i|/n-p_r| = |\E\one_{\{Z_i>1/2\}}-\E Z_i+ \E Z_i-p_r|\le \E\min\{Z_i, 1-Z_i\} + \rho <2\rho.
\end{equation}
Note that we have $\frac{|A_i\setminus W_i|}n + \frac{|(V_i\setminus A_i)\setminus X_i|}n = 1-\frac{|W_i|+|X_i|}{n} = \Pp(Z_i \in (\xi, 1-\xi))$.
By Markov's inequality, we get
\[
 \frac{|A_i\setminus W_i|}n + \frac{|(V_i\setminus A_i)\setminus X_i|}n= \Pp(\min\{Z_i,1-Z_i\} > \xi) \le\frac\rho\xi,
\]
together with \eqref{ine:a_i} and $2\rho+\rho/\xi<\xi$, all asserted size estimates follow.
\end{proof}

\section{Hypergraphs not close to \texorpdfstring{$H_k^0(n;n)$}{the extremal construction}}\label{sec:nonclose}

In this section, we prove the non-close case.  We shall use the
following rainbow fractional matching theorem of Aharoni, Holzman and
Jiang~\cite{AHJ}.

\begin{theorem}[Aharoni, Holzman and Jiang, \cite{AHJ}]\label{thm:ahj}
Let $r\ge2$ be an integer and let $x$ be a positive rational number.
Let $G_1,\ldots,G_{\lceil rx\rceil}$ be $r$-graphs satisfying
$\nuast(G_i)\ge x$ for every $i\in[\lceil rx\rceil]$.  Then there are
edges $e_i\in E(G_i)$, one for each $i\in[\lceil rx\rceil]$, such that
the $r$-graph with edge set
$\{e_1,\ldots,e_{\lceil rx\rceil}\}$ has a fractional matching of
size $x$.
\end{theorem}

The next lemma provides the absorbing matching used in the argument.

\begin{lemma}[Lo and Markstr\"om, \cite{LM}]\label{thm:absorber}
Fix \( k \ge 2 \), \( 1 \le \ell < k \), and \( 0 < \alpha < 1/(10k^3) \), and set \( \alpha' = \alpha^{2k-1}/20 \). Then for all sufficiently large \( n \), every balanced \( k \)-partite \( k \)-graph \( H \) satisfying 
\[
\delta_\ell(H) \ge \left( \frac{1}{2} + \alpha \right) n^{k-\ell}
\] 
contains a matching \( M_{\rm abs} \) of size at most \( (k-1)\alpha^k n \) such that for every balanced subset \( W \subseteq V(H) \setminus V(M_{\rm abs}) \) with \( |W| \le k\alpha' n \), the induced subgraph \( H[V(M_{\rm abs}) \cup W] \) admits a perfect matching.
\end{lemma}

To obtain an almost perfect matching in a sparse quasi-regular
hypergraph, we use the following theorems of Pippenger and Spencer~\cite{PS}.

\begin{theorem}[Pippenger and Spencer,~\cite{PS}]\label{thm:fr}
For every integer $k\ge2$ and every $\sigma>0$, there exist $\tau>0$ and $D_0$ such that the following holds.  Let $J$ be an $n$-vertex $k$-graph and let $D$ satisfy $n\ge D\ge D_0$.  
If for each vertex $v\in V(J)$, we have $d(v)=(1\pm\tau)D$
and for every pair $\{u,v\}$ of vertices of $J$, we have
$d(u,v)<\tau D$,
then $J$ contains a matching covering all but at most $\sigma n$ vertices.
\end{theorem}

Now we convert the perfect fractional matching(s) from Section 2 to an almost perfect matching.

\begin{lemma}
[Almost perfect matching]
\label{lem:almostcover}
For every integer \(k\ge4\) and every \(\varepsilon>0\),
there exists \(\gamma=\gamma(k,\varepsilon)>0\) such that,
for every \(\eta>0\), the following holds for all sufficiently
large \(n\).
Let $H$ be a balanced $k$-partite $k$-graph with vertex classes $V_1,\ldots,V_k$, each of size $n$. Suppose that $\delta_1(H)\ge (c_k-\gamma)n^{k-1}$ and that $H$ is not $\varepsilon$-contained in $H_k^0(n;n)$.
Then $H$ contains a matching $M$ covering all but at most $\eta n$ vertices in each vertex class.
\end{lemma}

\begin{proof}
Let $1/n\ll \gamma\ll \gamma_0\ll \varepsilon,1/k$ and $1/n\ll \tau\ll \eta,1/k$.
Apply \Cref{thm:fractional-nonclose} with closeness parameter $\varepsilon$, and $\gamma_0$ in place of $\gamma$.
Set $t:=\lfloor n/\log n\rfloor$. 
We construct perfect fractional matchings $f_1,\ldots,f_t$ in $H$ such that, writing $E_i:=\{e\in E(H):f_i(e)>0\}$, the following properties hold:
\begin{enumerate}[label=(\roman*),font=\normalfont]
\item\label{v1} $|E_i|\le kn$ for every $i\in[t]$, \item\label{v2} $\sum_{i=1}^t\sum_{e\supseteq D}f_i(e)\le k$ for every pair $D\subseteq V(H)$, \item\label{v3} $E_1,\ldots,E_t$ are pairwise disjoint.
\end{enumerate}
We proceed inductively, starting with the empty family.
Suppose that $f_1,\ldots,f_s$ have already been constructed for some $s<t$ and satisfy \ref{v1}--\ref{v3}. 
For every pair $D\subseteq V(H)$, define $L_s(D):=\sum_{i=1}^s\sum_{e\supseteq D}f_i(e)$ and put $\mathcal U_s:=\{D\in\binom{V(H)}2:L_s(D)>k-1\}$.
Let $E'_s$ consist of all edges containing some pair in $\mathcal U_s$, and set $F_s:=E'_s\cup E_1\cup\cdots\cup E_s$ and $G_s:=H-F_s$.
We first show that only $o(n^{k-1})$ edges incident with any fixed vertex are deleted. 
Since each $f_i$ is a perfect fractional matching where $i\in[s]$, for every $v\in V(H)$ we have $\sum_{w\ne v}L_s(\{v,w\})=\sum_{i\in[s]}\sum_{w\ne v}\sum_{\{v,w\}\subseteq e}f_i(e)=(k-1)s$.
Hence at most $s$ pairs in $\mathcal U_s$ contain $v$, and these pairs account for at most $sn^{k-2}$ edges through $v$.
Moreover, $\sum_{D\in\binom{V(H)}2}L_s(D)=\binom{k}{2}sn$.
Since every pair in $\mathcal U_s$ has load larger than $k-1$, we obtain $|\mathcal U_s|\le \binom{k}{2}sn/(k-1)=ksn/2$.
Thus the pairs in $\mathcal U_s$ not containing $v$ account for at most $ksn^{k-2}/2$ edges through $v$.
Finally, by \ref{v1}, $d_{E_1\cup\cdots\cup E_s}(v)\le ksn$.
Consequently, $d_{F_s}(v)\le sn^{k-2}+ksn^{k-2}/2+ksn \le2ksn^{k-2}$. 
It follows that $\delta_1(G_s)\ge(c_k-\gamma_0)n^{k-1}$ since $1/n\ll\gamma\ll\gamma_0$.

Furthermore, $G_s$ is not $\varepsilon$-contained in $H_k^0(n;n)$. 
Otherwise, since $G_s\subseteq H$, the hypergraph $H$ would also be $\varepsilon$-contained in the same copy of $H_k^0(n;n)$, a contradiction.
Hence, \Cref{thm:fractional-nonclose} implies that $G_s$ has a perfect fractional matching and, in particular, $\nu^*(G_s)=n$. Applying \Cref{thm:ahj} to $kn$ copies of $G_s$, we obtain a fractional matching $f_{s+1}$ of size $n$ whose support $E_{s+1}:=\{e\in E(G_s):f_{s+1}(e)>0\}$ has size at most $kn$. 
Since $G_s$ has $kn$ vertices, $f_{s+1}$ is in fact a perfect fractional matching. 
We now verify \ref{v1}--\ref{v3} for $f_1,\ldots,f_{s+1}$.

\textbf{Verification of \ref{v1}.}
By the choice of $f_{s+1}$ from \Cref{thm:ahj}, its support satisfies $|E_{s+1}|\le kn$. 
Together with the induction
hypothesis, this proves \ref{v1}.
\textbf{Verification of \ref{v2}.}
Fix a pair $D\subseteq V(H)$. If $D\in\mathcal U_s$, then every edge containing $D$ belongs to $E'_s$, and hence no edge of $G_s$ contains $D$. Since $f_{s+1}$ is supported on $G_s$, we have $\sum_{e\supseteq D}f_{s+1}(e)=0$.
Therefore $L_{s+1}(D)=L_s(D)\le k$, where the last inequality follows from the induction hypothesis \ref{v2}.
If $D\notin\mathcal U_s$, then $L_s(D)\le k-1$. 
Since $f_{s+1}$ is a fractional matching, $\sum_{e\supseteq D}f_{s+1}(e)\le1$.
Hence $L_{s+1}(D)\le(k-1)+1=k$.
Thus \ref{v2} holds.
\textbf{Verification of \ref{v3}.}
Since $E_{s+1}\subseteq E(G_s)$ and $G_s=H-F_s$, while $E_1\cup\cdots\cup E_s\subseteq F_s$, we have $E_{s+1}\cap(E_1\cup\cdots\cup E_s)=\varnothing$.
Together with the induction hypothesis, this proves that $E_1,\ldots,E_{s+1}$ are pairwise disjoint, and hence \ref{v3} holds.

Now define $h(e):=\sum_{i=1}^t f_i(e)$.
Since $E_1,\ldots,E_t$ are pairwise disjoint, we have $0\le h(e)\le1$ for every edge $e$.
Construct a random spanning subgraph $J\subseteq H$ by including every edge $e$ independently with probability $h(e)$.
For every vertex $v$, $\mathbb E[d_J(v)]=\sum_{e\ni v}h(e)=t$, because every $f_i$ is perfect. 
For every pair $D$, $\mathbb E[d_J(D)]=\sum_{e\supseteq D}h(e)\le k$.
By Chernoff bound, we obtain that, with high probability, $(1-1/\log n)t\le d_J(v)\le(1+1/\log n)t$ holds for every $v\in V(H)$, and $\Delta_2(J)\le\sqrt n$.
Indeed, for a fixed vertex the failure probability is $\exp(-\Omega(n/\log^3 n))$, whereas for a fixed pair it is at most $(ek/\sqrt n)^{\sqrt n}$. 
A union bound over the $kn$ vertices and $O_k(n^2)$ pairs gives the desired
simultaneous estimates.

Apply \Cref{thm:fr} with parameter $\eta$. 
Let $\tau>0$ and $D_0$ be the corresponding constants. 
Since $t\sim n/\log n$, for sufficiently large $n$ we have $t\ge D_0$, $1/\log n<\tau$, and $\sqrt n<\tau t$.
Thus \Cref{thm:fr} gives a matching $M$ in $J$, and hence in $H$, which leaves at most $\eta kn$ vertices uncovered.
Since $H$ is balanced and every edge contains exactly one vertex from each class, the uncovered set is balanced.
Consequently at most $\eta n$ vertices are uncovered in each class, and therefore $|M|\ge(1-\eta)n$.
\end{proof}

\begin{theorem}\label{thm:nonclose}
Fix $k\ge4$ and $\varepsilon>0$.  For all sufficiently large $n$, every balanced $k$-partite $k$-graph $H$ satisfying
\[
 \delta_1(H)>d_k(n,n-1)
\]
and not $\varepsilon$-contained in $H_k^0(n;n)$ has a perfect matching.
\end{theorem}

\begin{proof}
Let $c_k:=1-\left(\frac{k-1}{k}\right)^{k-1}$.
Apply \Cref{lem:almostcover} with closeness parameter $\varepsilon/2$.
By decreasing $\gamma$ if necessary, assume that $\gamma<\min\{c_k/2,c_k-1/2\}$.
Choose constants satisfying $1/n\ll \eta\ll \alpha'\ll \alpha\ll \gamma\ll \varepsilon,1/k$, where $\alpha':=\alpha^{2k-1}/20$.

By \eqref{eq:dk-general}, we have $d_k(n,n-1)=c_kn^{k-1}+O_k(n^{k-2})$.
Hence, $d_k(n,n-1)\ge(c_k-\gamma/10)n^{k-1}>(1/2+\alpha)n^{k-1}$.
Since $\delta_1(H)>d_k(n,n-1)$, \Cref{thm:absorber} with $\ell=1$ yields an absorbing matching $M_{\rm abs}$ of size $m:=|M_{\rm abs}|\le(k-1)\alpha^k n$ such that every balanced set $W\subseteq V(H)\setminus V(M_{\rm abs})$ with $|W|\le k\alpha'n$ can be absorbed by $M_{\rm abs}$.

Let $H':=H-V(M_{\rm abs})$ and let $N:=n-m$.
Since $M_{\rm abs}$ is a matching, $H'$ is a balanced $k$-partite $k$-graph with each vertex class of size $N$.

We first verify the minimum-degree condition required by \Cref{lem:almostcover}. Each vertex of $H'$ loses at most $(k-1)mn^{k-2}$ incident edges when $V(M_{\rm abs})$ is deleted. 
Therefore, $\delta_1(H')\ge\delta_1(H)-(k-1)mn^{k-2}$.
Using the lower bound on $d_k(n,n-1)$ and $m\le(k-1)\alpha^k n$, we obtain 
\[
\delta_1(H')\ge(c_k-\gamma/10)n^{k-1}-(k-1)^2\alpha^k n^{k-1}\ge(c_k-\gamma/5)n^{k-1}>(c_k-\gamma)N^{k-1}.
\]

We next claim that $H'$ is not $(\varepsilon/2)$-contained in $H_k^0(N;N)$. Suppose otherwise, then there are sets $W_i'\subseteq V_i(H')$, with $|W_i'|=d_i(N)$ for every $i\in[k]$, determining a copy $H_0'$ of $H_k^0(N;N)$ such that $|E(H_0')\setminus E(H')|\le(\varepsilon/2)N^k$.

For each $i\in[k]$, extend $W_i'$ to a set $W_i^0\subseteq V_i(H)$ with $|W_i^0|=d_i(n)$, using vertices of $V(M_{\rm abs})$. 
This is possible since $0\le d_i(n)-d_i(N)\le m$.
Let $H_0$ be the copy of $H_k^0(n;n)$ determined by $W_1^0,\ldots,W_k^0$. 
By construction, the restriction of $H_0$ to $V(H')$ is precisely $H_0'$.
Hence every edge of $E(H_0)\setminus E(H)$ either lies entirely in $V(H')$, in which case it belongs to $E(H_0')\setminus E(H')$, or intersects $V(M_{\rm abs})$.
Since $|V(M_{\rm abs})|=km$, there are at most $kmn^{k-1}$ edges of the latter type.
Consequently, $|E(H_0)\setminus E(H)|\le(\varepsilon/2)N^k+kmn^{k-1}$.
Using $N\le n$ and $m\le(k-1)\alpha^k n$, we obtain $|E(H_0)\setminus E(H)|
\le(\varepsilon/2)n^k+k(k-1)\alpha^k n^k<\varepsilon n^k$, contradicting the assumption that $H$ is not $\varepsilon$-contained in $H_k^0(n;n)$. 
Thus $H'$ is not $(\varepsilon/2)$-contained in $H_k^0(N;N)$.

We may therefore apply \Cref{lem:almostcover} to $H'$ with parameters $\varepsilon/2$ and $\eta=\alpha'/2$. 
Thus $H'$ contains a matching $M_{\rm near}$ covering all but at most $\eta N$ vertices in each vertex class.

Let $W:=V(H')\setminus V(M_{\rm near})$.
Since $H'$ is balanced and every edge of $M_{\rm near}$ contains exactly one vertex from each class, the set $W$ is balanced.
Moreover, $|W|\le k\eta N=k\alpha'N/2<k\alpha'n$.
By the absorbing property of $M_{\rm abs}$, $H[V(M_{\rm abs})\cup W]$ contains a perfect matching $M_{\rm fill}$. 
Since $M_{\rm near}$ is disjoint from $V(M_{\rm abs})\cup W$, the union $M_{\rm near}\cup M_{\rm fill}$ is a perfect matching of $H$.
\end{proof}

\section{Proof of \Cref{thm:main}}\label{sec:completion}

\begin{lemma}[Lu, Wang and Yuan, \cite{LWY}]\label{lem:close}
\label{thm:close}
Fix $k\ge4$.  There is $\varepsilon_k>0$ such that, for all sufficiently
large $n$, every balanced $k$-partite $k$-graph $H$ which is
$\varepsilon_k$-contained in $H_k^0(n;n)$ and satisfies $\delta_1(H)>d_k(n,n-1)$
has a perfect matching.
\end{lemma}

\begin{proof}[Proof of \Cref{thm:main}]
The construction $H'_k(n;n-1)$ has minimum vertex degree
$d_k(n,n-1)$ and no perfect matching, so $m'_1(k,n)\ge d_k(n,n-1)+1$.
For the reverse inequality, fix $\varepsilon=\varepsilon_k$ from
\Cref{thm:close}, and let $H$ be a balanced $k$-partite $k$-graph with
$\delta_1(H)>d_k(n,n-1)$.
If $H$ is $\varepsilon$-contained in $H_k^0(n;n)$, then apply \Cref{lem:close}.  
If it is not, then apply \Cref{thm:nonclose}.  In either case $H$ has a perfect matching.  Since vertex degrees are integers,
this proves $m'_1(k,n)=d_k(n,n-1)+1$.
\end{proof}

\medskip
\noindent
\textbf{Declaration of the use of generative AI.}
Generative AI (ChatGPT) was used to assist with manuscript preparation and language polishing. All proof ideas are entirely due to the authors, who take full responsibility for the accuracy and content of this work.

\begingroup
\raggedright

\endgroup

\end{document}